\documentclass[a4,11pt]{amsart}

\usepackage{latexsym}
\usepackage{amsmath}
\usepackage{amssymb}
\usepackage{amsthm}
\usepackage[mathcal]{eucal}
\usepackage{amsmath, amssymb,enumerate}
\usepackage{mathrsfs}
\usepackage{enumerate}
\usepackage[colorlinks=true, pdfstartview=FitV, linkcolor=blue, citecolor=blue,
 urlcolor=blue]{hyperref}
\usepackage{color}
\allowdisplaybreaks

 \makeatletter \@addtoreset{equation}{section}

\makeatother \makeatletter

\newtheorem{theorem}{Theorem}
\newtheorem{proposition}{Proposition}
\newtheorem{lemma}{Lemma}
\newtheorem{remark}{Remark}

\newtheorem{coro}{Corollary}

\newcommand{\R}{\mathbb{R}}

\newcommand{\N}{\mathbb{N}}

\newcommand{\eps}{\varepsilon}

\allowdisplaybreaks
\begin{document}
\title[Kernel estimates for elliptic operators]{Kernel estimates for elliptic operators with unbounded diffusion, drift and potential terms}
\author{S.E. Boutiah}
\address{Department of Mathematics, University Ferhat Abbas Setif-1, Setif 19000, Algeria.}
\email{sallah\_eddine.boutiah@yahoo.fr; \,\,\,boutiah@univ-setif.dz}
\author{A. Rhandi}
\address{Dipartimento di Ingegneria dell'Informazione, Ingegneria Elettrica e Matematica Applicata, Universit\`a degli
Studi di Salerno, Via Giovanni Paolo II, 132, I 84084 FISCIANO (Sa), Italy.}
\email{arhandi@unisa.it}
\author{C. Tacelli}
\address{Dipartimento di Ingegneria dell'Informazione, Ingegneria Elettrica e Matematica Applicata, Universit\`a degli
Studi di Salerno, Via Giovanni Paolo II, 132, I 84084 FISCIANO (Sa), Italy.}
\email{ctacelli@unisa.it}
\subjclass[2000]{35K08, 35J10, 47D08, 35K20, 47D07}
\keywords{Schr\"odinger type operator, semigroup, heat kernel estimates.}

\maketitle

\begin{abstract}
In this paper we prove that the heat kernel $k$ associated to the operator
$A:= (1+|x|^\alpha)\Delta +b|x|^{\alpha-1}\frac{x}{|x|}\cdot\nabla -|x|^\beta$ satisfies
$$
k(t,x,y) \leq c_1e^{\lambda_0 t+ c_2t^{-\gamma}}\left(\frac{1+|y|^\alpha}{1+|x|^\alpha}\right)^{\frac{b}{2\alpha}}
\frac{(|x||y|)^{-\frac{N-1}{2}-\frac{1}{4}(\beta-\alpha)}}{1+|y|^\alpha}
e^{-\frac{\sqrt{2}}{\beta-\alpha+2}\left(|x|^{\frac{\beta-\alpha+2}{2}}+ |y|^{\frac{\beta-\alpha+2}{2}}\right)}
$$
for $t>0,\,|x|,\,|y|\ge 1$, where $b\in\mathbb{R}$, $c_1,\,c_2$ are positive constants, $\lambda_0$ is the largest eigenvalue of the operator $A$, and $\gamma=\frac{\beta-\alpha+2}{\beta+\alpha-2}$, in the case where $N>2,\,\alpha>2$ and $\beta>\alpha -2$. The proof is based on the relationship between the log-Sobolev inequality and the ultracontractivity of a suitable semigroup in a weighted space.
\end{abstract}


\section{Introduction}
Consider the following elliptic operator
\begin{equation}\label{def-A}
A_{b,c}u(x)=(1+|x|^{\alpha})\Delta u(x) + b|x|^{\alpha-1}\frac{x}{|x|}\cdot\nabla - c|x|^{\beta}u(x), \quad x\in \R^N,
\end{equation}
where $\alpha>2$, $\beta>\alpha-2$, $b\in \R$ and $c>0$. For simplicity we denote by $A$ the operator $A_{b,1}$.

The aim of this paper is to study the behaviour of the heat kernel associated to the operator $A$. As a consequence one obtains precise estimates for the eigenfunctions associated to $A$.
In recent years the interest in second order elliptic operators with polynomially growing coefficients and their associated semigroups increases considerably, see for example
\cite{AC-AR-CT}, \cite{AC2-AR2-CT2}, \cite{AC-CT}, \cite{F-L}, \cite{G-S}, \cite{G-S 3}, \cite{Luca - Abde}, \cite{Me-Sp-Ta}, \cite{Me-Sp-Ta2}, \cite{Lo-Be} and the references therein.

For the operator $A_{b,c}$ it is proved in \cite{boutiah et al} that
the $L^p$-realization of $A_{b,c}$ in $L^p(\R^N)$ for $1<p<\infty$ with domain
$$
D_p(A)=\{u\in W^{2,p}(\R^N)\;|\; (1+|x|^\alpha)|D^2u|,(1+|x|^\alpha)^{1/2}\nabla u,|x|^\beta u\in L^p(\R^N) \}
$$
 generates a strongly continuous and analytic semigroup $T_p(\cdot)$ for $\alpha >2$, $\beta >\alpha-2$, $b\in \R$ and $c>0$, see also \cite{AC-AR-CT} for the case $b=0$. Moreover,
this semigroup is consistent, immediately compact and ultracontractive.\\
Furthermore, since the coefficients of the operator $A_{b,c}$ are locally regular, we know that the semigroup $T_p(\cdot)$
admits an integral representation through a heat kernel $k(t,x,y)$ , i.e.
$$T_p(t)f(x)=\int_{\R^N}k(t,x,y)f(y)\,dy,\quad t>0,\,x\in \R^N$$
  for all $f\in L^p(\R^N)$, cf. \cite{Lo-Be}, \cite{Me-Pa-Wa}.

In \cite{Luca - Abde} (resp. \cite{AC2-AR2-CT2}) estimates of the kernel $k(t,x,y)$
for $b=0$, $\alpha\in [0,2)$  and $\beta>2$ (resp. $b=0$, $\alpha>2$ and $\beta>\alpha-2$) were obtained. Even in the non-autonomous case, for a large class of second order elliptic operators with unbounded coefficients, heat kernel estimates was obtained, by using the techniques of Lyapunov functions, in \cite{KLR}.
For the critical case $\beta=\alpha -2$, estimates of the heat kernel associated to the operator $(1+|x|^{\alpha})\Delta -c|x|^{\alpha-2}$ for $\alpha>2$, $c>0$ and $N>2$ have been proved in \cite{TD-RM-CT}. Concerning the operator $A_{b,0}$, heat kernel estimates was obtained in
 \cite{Me-Sp-Ta}.

Our goal is to prove upper bounds for the operator $A:=A_{b,1}$ in the case where $\alpha> 2$, $\beta>\alpha-2$ and any constant $b \in \mathbb{R}$.
Our method consists in providing upper and lower estimates for the ground state $\Phi$ of $A$ corresponding to the largest eigenvalue $\lambda_0$ and adapting the arguments used in \cite{Davies1}.\\ \ \\
The paper is organized as follows.
In section \ref{eigenfunction} we
show that the eigenfunction $\Phi$ of $A$ associated to the largest eigenvalue $\lambda_0$ can be estimated from below and above by the function
$$|x|^{-\frac{N-1}{2}-\frac{\beta-\alpha}{4}} (1+|x|^{\alpha})^{-\frac{b}{2\alpha}}e^{-\int_1^{|x|}\sqrt{\frac{r^\beta}{1+r^\alpha}}dr}$$
for $|x|$ and $|y|$ sufficiently large.
\\
In Section \ref{form amu}, by means of a suitable multiplication operator $Tu=\phi u$, we rewrite the operator $A$ in the following form
\[
A=T^{-1}HT,
\]
where $H=(1+|x|^\alpha)\Delta -U$ with
$U=(1+|x|^{\alpha})\frac{\Delta \phi}{\phi}+|x|^\beta$
and use an associated positive, closed, symmetric form  $h(\cdot,\cdot)$ defined on a domain $D(h)$ in an appropriate weighted Hilbert space  $L_{\mu}^{2}(\mathbb{R}^N)$. This permits us to define the associated self-adjoint operator $H_\mu$ and his corresponding semigroup $(e^{tH_\mu})$. It can be seen that $(e^{tH_\mu})$ is given by a heat kernel $k_\mu$.
Adapting the arguments used in \cite{Davies1} and \cite{Luca - Abde}, we prove the following intrinsic ultracontractivity
\[
k_\mu(t,x,y)\leq c_1 e^{\lambda_0 t}e^{c_2 t^{-\gamma}}\Phi(x)\Phi(y), \quad t>0,\,x,\,y\in \R^N,
\]
where $c_1,c_2$ are positive constant, $\gamma=\frac{\beta-\alpha+2}{\beta+\alpha-2}$ and $\lambda_0$ is the largest eigenvalue of $A$, provided that $N>2,\,\alpha> 2$ and $\beta>\alpha-2$. Accordingly, we obtain upper bounds of the heat kernel
$$
k(t,x,y) \leq c_1e^{\lambda_0 t+ c_2t^{-\gamma}}\left(\frac{1+|y|^\alpha}{1+|x|^\alpha}\right)^{\frac{b}{2\alpha}}
\frac{(|x||y|)^{-\frac{N-1}{2}-\frac{1}{4}(\beta-\alpha)}}{1+|y|^\alpha}
e^{-\frac{\sqrt{2}}{\beta-\alpha+2}\left(|x|^{\frac{\beta-\alpha+2}{2}}+ |y|^{\frac{\beta-\alpha+2}{2}}\right)}
$$
for $t>0,\,|x|,|y|\ge 1$.\\

{\large \bf Notation.}
For $x\in \mathbb{R}^N$ and $r>0$ we set $B_r = \{ x\in \mathbb{R}^N : |x|<r\}$.
We denote by $<\cdot,\cdot>$ the euclidean scalar product and by $|\cdot|$  the euclidean norm.
We use as always a standard notation for function spaces. So, we denote by $L^p(\mathbb{R}^N)$ and $W^{2,p}(\mathbb{R}^N)$ the standard $L^p$ and Sobolev spaces, respectively. The space of bounded and continuous functions on $\mathbb{R}^N$ is denoted by $C_{b}(\mathbb{R}^N)$. Finally, in the whole manuscript the notation $\phi \approx \psi$ on a set $\Omega$ means that there are positive constants $C_1,\,C_2$ such that $C_1\psi(x)\le \phi(x)\le C_2\psi(x)$ for all $x\in \Omega$.

%
%
%
%

\section{Estimating the ground state $\Phi$}\label{eigenfunction}
For $\alpha> 2,\,\beta>\alpha-2$ and $N>2$, we denote by $A_p$ the realization in $L^p(\R^N),\,1<p<\infty$, of the operator $A:=A_{b,1}$ defined in \eqref{def-A}. We recall, see \cite[Theorem 3 and Theorem 4]{boutiah et al}, that $A_p$ with domain
$$
D_p(A)=\{u\in W^{2,p}(\R^N)\;|\; (1+|x|^\alpha)|D^2u|,(1+|x|^\alpha)^{1/2}\nabla u,|x|^\beta u\in L^p(\R^N) \}
$$
generates a strongly continuous and analytic semigroup $T_p(\cdot)$ in $L^p(\R^N)$. Moreover, for $t>0$, $T_p(t)$ maps $L^p(\R^N)$ into $C_b^{1+\eta}(\R^N)$ for any $\eta \in (0,1)$, see \cite[Proposition 5]{boutiah et al}, and the semigroup $T_p(\cdot)$ is immediately compact, see \cite[Proposition 6]{boutiah et al}. As a consequence one obtains that the spectrum $\sigma(A_p)$ of $A_p$ consists of a sequence of negative real eigenvalues which accumulates at $-\infty$, and $\sigma(A_p)$ is independent of $p$. As in \cite{AC-AR-CT} and \cite{Luca - Abde} one can see that
$T_p(\cdot)$ is irreducible, the eigenspace corresponding to the largest eigenvalue
   $\lambda_0$ of $A_p$ is one-dimensional and is spanned by a strictly positive functions $\Phi$,
which is radial, belongs to $C_b^{1+\nu}(\R^N)\cap C^2(\R^N)$ for any $\nu \in (0,1)$ and tends to $0$ as $|x|\to \infty$.

In this section we prove precise estimates for the eigenfunction $\Phi$.
The technique used here is inspired by the work \cite{AC2-AR2-CT2}.

Since $\Phi$ is radial,
one has to analyze the asymptotic behavior of the solutions to an ordinary differential equation. In this context
some ideas coming from the Wentzel-Kramers-Brillouin (or Liouville-Green) approximation will be of great help. For more details see \cite{Olver}.

\begin{proposition}\label{baondedofphi}
Let $\lambda_0<0$ be the largest eigenvalue of $A$ and $\Phi$ be the
corresponding eigenfunction. If $N>2,\,\alpha >2$ and $\beta>\alpha-2$ then
$$
\Phi \approx |x|^{-\frac{N-1}{2}-\frac{1}{4}(\beta-\alpha) }(1+\vert x\vert^{\alpha})^{-\frac{b}{2\alpha}}
    e^{-\int_1^{|x|}\sqrt{\frac{r^\beta}{1+r^\alpha}}dr}
$$
on $\R^N\setminus B_1$.
\end{proposition}
\begin{proof}
Let $g_{\alpha,\beta,\lambda}$ be the function defined as
\begin{equation}\label{g}
g_{\alpha,\beta,\lambda}(x)=|x|^{-\frac{N-1}{2}}(1+\vert x\vert^{\alpha})^{-\frac{b}{2\alpha}}\mathfrak{h}^{-\frac{1}{4}}(|x|)
\exp \left\{-\int_1^{|x|}\mathfrak{h}^{\frac{1}{2}}(s)ds -\int_1^{|x|}\mathfrak{v}_\lambda(s)ds \right\}\;,
\end{equation}
where
$\lambda \in \R ,\,\mathfrak{h}(r)=\frac{r^\beta}{1+r^\alpha}$,
and $\mathfrak{v}_\lambda$ is a smooth function to be chosen later on.
If we set
\begin{equation}\label{w}
w(r)=r^{\frac{N-1}{2}}(1+r^{\alpha})^{\frac{b}{2\alpha}}g_{\alpha,\beta,\lambda}(r),
\end{equation}
the calculation of $w'$ gives us
\begin{align*}
w'(r) = & \left( \frac{N-1}{2}\right)r^{\frac{N-1}{2}-1}(1+r^{\alpha})^{\frac{b}{2\alpha}}g_{\alpha,\beta,\lambda}(r)\\
&+ \frac{b}{2}r^{\frac{N-1}{2}+\alpha-1}(1+r^{\alpha})^{\frac{b}{2\alpha}-1}g_{\alpha,\beta,\lambda}(r) + r^{\frac{N-1}{2}}(1+r^{\alpha})^{\frac{b}{2\alpha}}g'_{\alpha,\beta,\lambda}(r).
\end{align*}

So, one obtains
\begin{equation}\label{w''1}
 w'=w\left( -\frac{\mathfrak{h}'}{4\mathfrak{h}}-\mathfrak{h}^{\frac{1}{2}}-\mathfrak{v}_\lambda\right)
\qquad
\hbox{and}\qquad
w''= w(g_1+m+\mathfrak{h}),
\end{equation}
where
\begin{equation}\label{R}
g_1=\frac{5}{16}\left(\frac{\mathfrak{h}'}{\mathfrak{h}}\right)^2-
\frac{\mathfrak{h}''}{4h}+\mathfrak{v}_\lambda^2+\mathfrak{v}_\lambda\left(\frac{\mathfrak{h}'}{2\mathfrak{h}}+
2\mathfrak{h}^{\frac{1}{2}}\right)-\mathfrak{v}_\lambda '-m
\end{equation}
and
$$
m(r):=\frac{(N-1)(N-3)}{4r^2} + \frac{b}{2}\left(N-2 + \alpha\right)\frac{r^{\alpha-2}}{1+r^{\alpha}} + \frac{b\alpha}{2}\left(\frac{b}{2\alpha}-1\right)\left(\frac{r^{\alpha-1}}{1+r^{\alpha}}\right)^2.
$$
On the other hand, by computing directly
 the second derivative of (\ref{w}) one has
\begin{align*}
&w''(r)=r^{\frac{N-1}{2}}(1+r^{\alpha})^{\frac{b}{2\alpha}}\left( g_{\alpha,\beta,\lambda}'' +
\left( \frac{N-1}{r} + b\frac{r^{\alpha-1}}{1+r^{\alpha}}\right) g_{\alpha,\beta,\lambda}'\right)  \\
+ &  r^{\frac{N-1}{2}}(1+r^{\alpha})^{\frac{b}{2\alpha}} \left[\frac{(N-1)(N-3)}{4r^2} + \frac{b}{2}\left(N-2 + \alpha\right)\frac{r^{\alpha-2}}{1+r^{\alpha}} + \frac{b\alpha}{2}\left(\frac{b}{2\alpha}-1\right)\left(\frac{r^{\alpha-1}}{1+r^{\alpha}}\right)^2\right] g_{\alpha,\beta,\lambda}.
\end{align*}
So, comparing with (\ref{w''1}) we get
\[
g_{\alpha,\beta,\lambda}'' +
\left( \frac{N-1}{r} + b\frac{r^{\alpha-1}}{1+r^{\alpha}}\right) g_{\alpha,\beta,\lambda}'=
\frac{r^\beta}{1+r^\alpha}g_{\alpha,\beta,\lambda}+g_1g_{\alpha,\beta,\lambda}.
\]
That is
\begin{equation}\label {Delta}
\Delta g_{\alpha,\beta,\lambda}(x) + b\frac{\vert x\vert^{\alpha-2}x}{1+|x|^\alpha}\cdot \nabla g_{\alpha,\beta,\lambda}(x) -\frac{|x|^\beta}{1+|x|^\alpha} g_{\alpha,\beta,\lambda}(x)
 =g_1(|x|)g_{\alpha,\beta,\lambda}(x).
\end{equation}
To evaluate the function $g_1$ we set $\xi=\frac{\beta-\alpha}{2}+1$, which is positive thanks to the assumption $\beta>\alpha-2$.
We have
\begin{align*}
\frac{\mathfrak{h}'}{\mathfrak{h}}& =  \frac{\beta r^{-1}(1+r^{\alpha}) - \alpha r^{\alpha-1}}{1+r^{\alpha}} \\ 
&=\frac{\beta-\alpha}{r} + \frac{1}{r}\left(\frac{\alpha}{1+r^{\alpha}} \right) =  \frac{1}{r}(\beta-\alpha)+\frac{1}{r}O(r^{-\alpha}).
\end{align*}
By the same argument we obtain
\[
\frac{\mathfrak{h}''}{\mathfrak{h}}=\frac{1}{r^2}(\beta-\alpha)(\beta-\alpha-1)+\frac{1}{r^2}O(r^{-\alpha}).
\]

Then (\ref{R}) is reduced to
\begin{eqnarray}\label{eq:errore-R2}
g_1(r) &=& -\mathfrak{v}_\lambda '+\frac{\mathfrak{v}_\lambda}{r}
    \left( \xi-1+O(r^{-\alpha})+2r^\xi \sqrt{\frac{r^\alpha}{1+r^\alpha}}\right)+\mathfrak{v}_\lambda^2 \nonumber \\
    & & +\frac{c_0}{r^2}+\frac{1}{r^2}\left(O(r^{-\alpha})+O(r^{-2\alpha})\right)\nonumber  + \frac{b}{2}\left(N-2 + \alpha\right)
    \frac{1}{r^2(1+r^\alpha)}
   \nonumber \\
  & & + \frac{b\alpha}{2}\left(\frac{b}{2\alpha}-1\right)
  \frac{1+2r^\alpha}{r^2(1+r^\alpha)^2} \nonumber \\
&=& -\mathfrak{v}_\lambda '
+\frac{\mathfrak{v}_\lambda}{r}\left(
	\xi-1+O(r^{-\alpha})+2r^\xi -2r^{\xi} \frac{  (1+r^\alpha)^{1/2}-r^{ \alpha/2 }   }{  (1+r^\alpha)^{1/2} }
	\right)\nonumber \\
&&+\mathfrak{v}_\lambda^2 +\frac{c_0}{r^2}+\frac{1}{r^2}O(r^{-\alpha})\nonumber
 \nonumber \\
&=& -\mathfrak{v}_\lambda '
+\frac{\mathfrak{v}_\lambda}{r}\left( \xi-1+2r^\xi +(1+r^\xi)O(r^{-\alpha})\right)
+\mathfrak{v}_\lambda^2\nonumber \\
& & +\frac{c_0}{r^2}+\frac{1}{r^2}O(r^{-\alpha}),
\nonumber \\
\end{eqnarray}
where
$$c_0=c_0(\xi)=\left( \frac{\xi-1}{2} \right)^2+\frac{\xi-1}{2}
-\frac{(N-1)(N-3)}{4}- \frac{b}{2}\left(N-2 + \alpha\right)- \frac{b\alpha}{2}\left(\frac{b}{2\alpha}-1\right) \nonumber \\ .$$ 
So, if we take in (\ref{eq:errore-R2})
\[
 \mathfrak{v}_\lambda(r)=\sum_ {i=1}^k c_i\frac{1}{r^{i\xi+1}},\quad r\ge 1,
\]
we obtain, as in the proof of \cite[Proposition 2.2]{AC2-AR2-CT2},
\begin{eqnarray*}
r^2g_1(r)
&=&\sum _{i=2}^{k-1}\left[c_i\xi(i+1)+2c_{i+1}+\sum_{j+s=i}c_jc_s \right]\frac{1}{r^{i\xi}}+(2c_1\xi+2c_2)r^{-\xi}\\
& &  +c_k\xi (k+1)\frac{1}{r^{k\xi}}+2c_1+\sum _{i+j\geq k}\frac{c_ic_j}{r^{(i+j)\xi}}
  +c_0+O(r^{-\alpha}),
\end{eqnarray*}
where $k\ge 3$ be chosen later. Again as in \cite{AC2-AR2-CT2},
 we can choose $c_1,\dots,c_{k}$ such that $$2c_1+c_0=\lambda ,\,2c_1\xi +2c_2=0
\hbox{\ and }\left[\xi(i+1)c_i+2c_{i+1}+\sum_{j+s=i}c_jc_s \right]=0$$ for $i=2,\cdots, k-1$
and obtain
\begin{align*}
r^2g_1(r) & =
\lambda+c_k\xi (k+1)\frac{1}{r^{k\xi}}+\sum_{i+j\ge k}\frac{c_ic_j}{r^{(i+j)\xi}}+ O(r^{-\alpha}).
\end{align*}
Hence,
\begin{align*}
g_1(r) & =O\left(\frac{1}{r^{k\xi+2} }\right)+O\left(\frac{1}{r^{\alpha +2} }\right)+\frac{\lambda}{r^2}.
\end{align*}

Since $\xi>0$, there exists a natural number $k\ge 3$ such that $k\xi+2-\alpha>0$. So, we have
\begin{eqnarray*}
(1+|x|^{\alpha}) \Delta g_{\alpha,\beta,\lambda}(x) &+& b\vert x\vert^{\alpha-2}x\cdot \nabla g_{\alpha,\beta,\lambda}(x)  -|x|^{\beta} g_{\alpha,\beta,\lambda}(x)\\
    &=& o(1)g_{\alpha,\beta,\lambda}(x)+\lambda \frac{1+|x|^{\alpha}}{|x|^2}g_{\alpha,\beta,\lambda}(x).
     \end{eqnarray*}
For $\Phi$ we know that
\begin{equation}\label{eq:eq-aut-diviso}
 \Delta \Phi + b\frac{|x|^{\alpha-2}}{1+|x|^{\alpha}}x\cdot \nabla \Phi -\frac{|x|^\beta}{1+|x|^\alpha}\Phi-\frac{\lambda_0}{1+|x|^\alpha}\Phi =0\;.
\end{equation}
Since $\alpha-2> 0$ and $\lambda_0<0$, for $|x|$ large enough we have
 $$o(1)+2\lambda_0 \frac{1+|x|^{\alpha}}{|x|^2} <\lambda_0.$$
Thus,
\begin{equation}\label{eq:f-2lambda-diviso}
\Delta g_{\alpha,\beta,2\lambda_0}(x) +b\frac{|x|^{\alpha-2}}{1+|x|^{\alpha}}x\cdot \nabla g_{\alpha,\beta,2\lambda_0}(x)  -\frac{|x|^{\beta}}{1+|x|^\alpha} g_{\alpha,\beta,2\lambda_0}(x)
   -\frac{\lambda_0}{1+|x|^\alpha} g_{\alpha,\beta,2\lambda_0}(x)< 0
\end{equation}
for all $x\in \R^N\setminus B_R$ for some $R>0$.
Comparing \eqref{eq:eq-aut-diviso} and \eqref{eq:f-2lambda-diviso}, in $\R^N\setminus B_R$
we have
\begin{eqnarray*}
\Delta (g_{\alpha,\beta,2\lambda_0}-C\Phi)+ b\frac{|x|^{\alpha-2}}{1+|x|^{\alpha}}x\cdot \nabla (g_{\alpha,\beta,2\lambda_0}(x)-C\Phi)
& < &
\frac{\lambda_0+|x|^\beta}{1+|x|^{\alpha}}(g_{\alpha,\beta,2\lambda_0}-C\Phi)
\end{eqnarray*}
 for any constant $C>0$. Since $\beta>0$, we deduce that
$$
\mathcal{W}(x):=\frac{\lambda_0+|x|^\beta}{1+|x|^{\alpha}}>0
$$
for $|x|$ large enough.
Note that both $g_{\alpha,\beta,2\lambda_0}$ and $\Phi$ go to $0$ as $|x|\to \infty$
and since
there exists $C_2$ such that $\Phi\le C_2 g_{\alpha,\beta,2\lambda_0}$ on $\partial B_R$, we can apply the maximum principle to the problem
$$\left\{\begin{array}{ll}
\left(\Delta + b\frac{|x|^{\alpha-2}}{1+|x|^{\alpha}}x\cdot \nabla\right) z(x)-\mathcal{W}(x)z(x)<0 \,\,\quad \hbox{\ in }\R^N\setminus B_R ,\\
z(x)\ge 0 \hspace*{3cm}\hbox{\ in }\partial B_R ,\\
\lim_{|x|\to \infty}z(x)=0,
\end{array} \right.
$$ where $z:=g_{\alpha,\beta,2\lambda_0}-C_2^{-1}\Phi$,
to obtain that
$\Phi\leq C_2 g_{\alpha,\beta,2\lambda_0}$ in $\R^N\setminus B_R$ (and by continuity in $\R^N\setminus B_1$). Here we apply the classical maximum principle on bounded domains,
since $\lim_{|x|\to \infty}z(x)=0$, cf. \cite[Theorem 3.5]{DG-NT}. Then,
\begin{align*}
\Phi(x) &\leq C_2 |x|^{-\frac{N-1}{2}-\frac{1}{4}(\beta-\alpha) }(1+\vert x\vert^{\alpha})^{-\frac{b}{2\alpha}}
\exp\left\{-\int_1^{|x|}\sqrt{\frac{r^\beta}{1+r^\alpha}}\,dr\right\}
\exp\left\{-\int_1^{|x|} \mathfrak{v}_{2\lambda_0}(r)dr\right\}\\
&\leq C_3 |x|^{-\frac{N-1}{2}-\frac{1}{4}(\beta-\alpha) }(1+\vert x\vert^{\alpha})^{-\frac{b}{2\alpha}}
\exp\left\{-\int_1^{|x|}\sqrt{\frac{r^\beta}{1+r^\alpha}}\,dr\right\},
\end{align*}
\begin{eqnarray}\label{extra-justi}
\lim_{|x|\to \infty}\int_1^{|x|} \mathfrak{v}_\lambda(r)\,dr &=&
\lim_{|x|\to \infty}\sum_{j=1}^k\frac{c_j}{j\xi}(1-|x|^{-j\xi})\nonumber \\
&=& \sum_{j=1}^k\frac{c_j}{j\xi}.
\end{eqnarray}

As regards lower bounds of $\Phi $, we observe that
\begin{align*}
 & \Delta g_{\alpha,\beta,0}(x) + b\frac{|x|^{\alpha-2}}{1+|x|^{\alpha}}x\cdot \nabla g_{\alpha,\beta,0}(x) -\frac{|x|^{\beta}}{1+|x|^\alpha}g_{\alpha,\beta,0}(x) \\
 & = \frac{o(1)}{1+|x|^\alpha}g_{\alpha,\beta,0}(x) > \frac{\lambda_0}{1+|x|^\alpha}g_{\alpha,\beta,0}(x)
\end{align*}
if $|x|\ge R$ for some suitable $R>0$.
Then,
\[
\Delta g_{\alpha,\beta,0}(x) + b\frac{|x|^{\alpha-2}}{1+|x|^{\alpha}}x\cdot \nabla g_{\alpha,\beta,0}(x)>
\frac{|x|^{\beta}}{1+|x|^\alpha}g_{\alpha,\beta,0}(x)+
\frac{\lambda_0}{1+|x|^\alpha}g_{\alpha,\beta,0}(x).
\]
Since $\Delta \Phi + b\frac{|x|^{\alpha-2}}{1+|x|^{\alpha}}x\cdot \nabla \Phi -\frac{|x|^\beta}{1+|x|^\alpha}\Phi-\frac{\lambda_0}{1+|x|^\alpha}\Phi =0\ $ we obtain
\begin{align*}
&\Delta (g_{\alpha,\beta,0}-\Phi) + b\frac{|x|^{\alpha-2}}{1+|x|^{\alpha}}x\cdot \nabla (g_{\alpha,\beta,0}(x) - \Phi) > \frac{|x|^\beta+\lambda_0}{1+|x|^\alpha}(g_{\alpha,\beta,0}-\Phi).
\end{align*}
Note that $|x|^\beta+\lambda_0$ is positive for $|x|\ge R$ and, arguing as above,
by the maximum principle and using \eqref{extra-justi} we have
\[
\Phi(x) \ge C_1g_{\alpha,\beta,0}(x)\geq C_1 |x|^{-\frac{N-1}{2}-\frac{1}{4}(\beta-\alpha) }(1+\vert x\vert^{\alpha})^{-\frac{b}{2\alpha}}
\exp\left\{-\int_R^{|x|}\sqrt{\frac{r^\beta}{1+r^\alpha}}\,dr\right\}
\]
for $|x|\ge R$. Since $0<\Phi \in C(\R^N)$, by modifying the constant $C_1$, we can see that, the above lower estimate of $\Phi$ remain valid for $1\le |x|\le R$.
\end{proof}
\begin{remark}
\begin{itemize}
\item[1.] If in the above proposition we take $b=0$, then we obtain exactly the upper and lower estimates for the ground state $\psi$ associated to the operator $(1+|x|^\alpha)\Delta -|x|^\beta$  established in \cite[Proposition 2.2]{AC2-AR2-CT2}.
\item[2.] If we denote by $\Phi_c$ the eigenfunction of the largest eigenvalue of $A_{b,c}$, then one can see that
$$\Phi_c\approx |x|^{-\frac{N-1}{2}-\frac{1}{4}(\beta-\alpha) }(1+\vert x\vert^{\alpha})^{-\frac{b}{2\alpha}}
    e^{-\int_1^{|x|}\sqrt{\frac{cr^\beta}{1+r^\alpha}}dr}
$$
on $\R^N\setminus B_1$.
\end{itemize}
\end{remark}
\section{Intrinsic ultracontractivity and heat kernel estimates}\label{form amu}

In this section we prove heat kernel estimates for $T_p(\cdot)$ through the relationship between
the log-Sobolev  inequality and the ultracontractivity of a suitable semigroup in a weighted $L^2$-space.


Consider the Hilbert spaces $L^2_\mu=L^2(\R^N,d\mu)$ with $d\mu (x)=\frac{1}{1+|x|^{\alpha}}dx$.
Define the function $\phi(x)=\left(1+|x|^\alpha\right)^{\frac{b}{2\alpha}}$ and the multiplication operator $T:\,L^2_{\phi^2\mu}\to L^2_\mu$ defined by $Tu=\phi u$.
The operator $A$
defined above can be written in the following way
\[
A=T^{-1}HT,
\]
where $H=(1+|x|^\alpha)\Delta -U$ and the potential
$U=(1+|x|^{\alpha})\frac{\Delta \phi}{\phi}+|x|^\beta$.
An easy computation gives us
\[
U=|x|^{\alpha-2}\frac{b}{2}\left( \frac{|x|^{\alpha}}{1+|x|^{\alpha}}\left( \frac{b}{2}-\alpha \right)+N+\alpha-2 \right)+|x|^\beta ,
\]
from which we can deduce that $U$ is bounded from below, since $\beta >\alpha -2$.

Since, for every $v\in C_c^\infty(\R^N)$, we have
$$\int_{\R^N}\left(|\nabla v|^2+\frac{\Delta \phi}{\phi}|v|^2\right)dx+\int_{\R^N}|x|^\beta |v|^2d\mu\\
=\int_{\R^N}\left| \nabla\left( \frac{v}{\phi} \right)\right|^2\phi^2dx  +\int_{\R^N}|x|^\beta|v|^2d\mu \geq 0,
$$
 we can associate to $H$ in $L^2_\mu$ the bilinear form $h$ defined by
\begin{equation}\label{formaquad}
h(v,w)=\int_{\R^N}\nabla v\cdot \nabla \overline w dx+\int_{\R^N}Uv\overline w d\mu
\end{equation}
on $D(h )=\overline{C_c^\infty(\R^N)}^{\|\cdot \|_{\mathcal H}}$ with $\mathcal H$ the Hilbert space
$$ {\mathcal H}=
\{v\in L^2_\mu \cap W_{loc}^{1,2}\left( \R^N \right): |U|^{1/2}v\in L^2_{\mu},\, \nabla v\in (L^2(\R^N))^N\}$$ endowed with the inner product
$$
\langle v,w\rangle_{{\mathcal H}}=\int_{\R^N}(1+U)v\overline{w}\,d\mu +\int_{\R^N}\nabla v\cdot \nabla \overline{w}\,dx.
$$
So, one can prove the following proposition
\begin{proposition}
The form $(h,D(h))$ is symmetric, continuous, closed and accretive.
\end{proposition}
Similarly to \cite[Lemma 9.1]{Me-Sp-Ta}  we have
\begin{lemma}
The injection of $D(h)$ in $L_\mu^2$ is compact.
\end{lemma}

Since the bilinear form $h$ is densely defined, accretive, continuous, closed and symmetric, one can associate the self-adjoint, dissipative operator $H_{\mu}$ defined by
\begin{eqnarray*}
\begin{array}{l}
D(H_{\mu})=\displaystyle\left\{v\in D(h) : \exists f\in L^2_\mu\,\, \mbox{ s.t. }
    h(v,w)=-\int_{\R^N}f\overline{w}\,d\mu,\,\mbox{ for every } w\in D(h)\right\},\\[3mm]
H_{\mu} v=f,
\end{array}
\end{eqnarray*}
see e.g. \cite[Prop. 1.24]{Ouhabaz}. Hence, $H_{\mu}$ generates a positive and analytic semigroup $(e^{tH_{\mu}})_{t\ge 0}$ in $L^2_\mu$, cf. \cite[Theorem 1.52 and Theorem 2.6]{Ouhabaz}.


The following result gives the relationship between $H_\mu$ and $A_p$.

\begin{lemma} \label{coerenza1} The following holds
\begin{equation}\label{DAmu}
D(H_{\mu})= \left\{u \in D(h) \cap W^{2,2}_{loc}(\R^N):
(1+|x|^\alpha) \Delta v -Uv \in L^2_\mu \right\}
\end{equation}
and $H_{\mu}v=(1+|x|^\alpha)\Delta v-U(x)v $ for $v \in D(H_{\mu})$. Moreover, if
$\lambda >\lambda '$ where $\lambda '$ is a suitable positive constant and
$f \in C_c^{\infty}(\R^N)$,
then
$$
T^{-1}(\lambda-H_{\mu})^{-1}Tf=(\lambda-A_2)^{-1} f.
$$
\end{lemma}

\begin{proof}
Let us begin by proving (\ref{DAmu}). The first inclusion $"\subseteq "$ is obtained by local elliptic regularity and \eqref{formaquad}.\\
For the second inclusion $"\supseteq"$ let us take $v\in D(h) \cap W^{2,2}_{loc}(\R^N)$ such that $f:=(1+|x|^\alpha)\Delta v-U(x)v\in L^2_\mu$.
Integrating by parts we obtain
\begin{equation}\label{eq:int-part}
h(v,w)=-\int f\overline{w}d\mu ,\quad \forall w\in C_c^{\infty}(\R^N).
\end{equation}
By the density of $C_c^{\infty}(\R^N)$ in $D(h)$, \eqref{eq:int-part} holds for every $w\in D(h)$.
This implies that $v \in D(H_{\mu})$.

To prove the coherence of the resolvents,
we consider, for a positive function $f \in C_c^\infty(\mathbb{R}^N)$, the following  elliptic problem
\begin{equation}\label{elliptic 1}
\quad
\begin{cases}
&\lambda u -Au =f \quad x \in B_n ,\\
& u = 0 \qquad x \in \partial B_n.
\end{cases}
\end{equation}
Since the operator $A$ is uniformly elliptic in the ball $B_n$, it is known that (\ref{elliptic 1})
admits a unique solution $u_{n}$ in $W^{2,2}(B_n)\cap W_{0}^{1,2}(B_n)$, (cf. \cite[Theorem 9.15]{DG-NT}).
Likewise, as in \cite{Lo-Be}, the sequence $(u_n)$ is increasing, positive and converges to a function $u$ in $D_{max}(A)$ satisfying $\lambda u-Au=f$, where
$$ D_{max}(A)= \{ u \in C_{b}(\mathbb{R}^N)\cap W_{loc}^{2,p}(\mathbb{R}^N) \quad\text{for all}\quad 1\le p<\infty : Au \in C_{b}(\mathbb{R}^N) \}.$$

Setting $v_n=Tu_n$ and $g=Tf$ we have, by \eqref{elliptic 1},
\begin{equation}\label{elliptic 1H}
\lambda v_n -Hv_n =g \quad \hbox{\ on }B_n.
\end{equation}
Moreover  $v_n\in W^{2,2}(B_n)\cap W_{0}^{1,2}(B_n)$.

Multiplying in \eqref{elliptic 1H}  by $\overline{w}\mu$ for $w\in W_0^{1,2}(B_n)$ and integrating by parts we obtain
\begin{equation}\label{variationel}
\lambda\int_{B_n}v_n\overline{w}d\mu +  \int_{B_n}\nabla v_n\cdot \nabla \overline{w}\,dx + \int_{B_n}Uv_n\overline{w}\,d\mu = \int_{B_n}g\overline{w}d\mu.
\end{equation}
In particular we obtain
\begin{equation}\label{variat2}
\lambda\int_{B_n}v_n^2d\mu +  \int_{B_n}|\nabla v_n|^2\,dx + \int_{B_n}Uv_n^2\,d\mu = \int_{B_n}gv_nd\mu.
\end{equation}
Since $\int_{B_n}|\nabla v_n|^2\,dx + \int_{B_n}Uv_n^2\,d\mu\geq 0$, it follows from (\ref{variat2}) that
$$\Vert v_n\Vert_{L_\mu^2} \le \frac{1}{\lambda}\Vert g \Vert_{L_\mu^2} $$
By the monotone convergence theorem, we deduce that $\lim_{n\to +\infty}v_n=v$ in $L_\mu^2$.
Furthermore since $U$ is bounded from below we can choose $\tilde{\lambda}$ such that $\lambda +U\geq 0$ for $\lambda>\tilde{\lambda}$. Then, by \eqref{variat2}, we have
$$\| \left(\lambda+U\right)^{\frac{1}{2}} v_n\Vert_{L_\mu^2}^2 \le \Vert g \Vert_{L_\mu^2}\Vert v_n\Vert_{L_\mu^2}\le \frac{1}{\lambda}\Vert g \Vert_{L_\mu^2}^2.$$
Choosing $\lambda>\lambda ':=\max\{-2\min U,\tilde{\lambda},0\}$, we obtain
$|U|\leq |\lambda +U|$, and hence $|U|^{\frac{1}{2}}v_n\rightarrow |U|^{\frac{1}{2}} v$ in $L_\mu^2$.

Similarly  one finds
$$
\Vert \nabla v_n \Vert_2^2 \le \Vert g \Vert_{L_\mu^2}\Vert v_n\Vert_{L_\mu^2} \le \frac{1}{\lambda}\Vert g \Vert_{L_\mu^2}^2, \quad \forall \lambda >\lambda '.
$$
It follows that there exists a suitable subsequence $(v_{k_n})$ of $(v_n)$ such that $\nabla v_{k_n}$
converges weakly. So, $v\in \mathcal{H}$ and $v$ belongs to the closure in $\mathcal{H}$ of $W^{1,2}$-functions with compact support, which implies that $v\in D(h)$. Letting now $n\to +\infty$ in \eqref{variationel} we obtain
$$h(v,w)=-\langle \lambda v-g,w\rangle _{L^2_\mu},\quad \lambda >\lambda ',$$
for all $w\in W^{1,2}$ having compact support, and hence for all $w\in D(h)$. Thus, $v\in D(H_\mu)$ and $\lambda v-H_\mu v=g$ for all $\lambda >\lambda '$.
Therefore, since $v=Tu$ and $g=Tf$, it follows that
$$(\lambda -H_\mu)^{-1}Tf=T(\lambda -A)^{-1}f$$ for all $f\in C_c^\infty(\R^N)$ and $\lambda >\lambda '$. So, the statement follows now from \cite[Theorem 2]{boutiah et al}.
%
%
%
\end{proof}

The lemma stated above implies in particular that
\[
e^{tH_{\mu}}f(x)=\int_{\R^N}k_\mu(t,x,y)f(y)\,d\mu(y),\quad f\in L^2_\mu
\]
with
\begin{equation}\label{k-mu}
k_\mu(t,x,y)=\phi(x)k(t,x,y)\phi(y)^{-1}(1+|y|^\alpha),\quad t>0,\,x,\,y\in \R^N.
\end{equation}

As an application of the Proposition \ref{baondedofphi}, we have
\begin{proposition}\label{on-diagonal}
If $N>2,\,\alpha > 2$ and $\beta >\alpha -2$, then
\begin{eqnarray*}
k(t,x,x)\ge Ce^{\lambda_0 t}\left(|x|^{-\frac{N-1}{2}-\frac{1}{4}(\beta-\alpha) }(1+\vert x\vert^{\alpha})^{-\frac{b}{2\alpha}}
    e^{-\int_1^{|x|}\sqrt{\frac{r^\beta}{1+r^\alpha}}dr}\right)^2(1+|x|^\alpha)^{\frac{b}{\alpha}-1},\quad t>0,
\end{eqnarray*}
for all $x\in\R^N\setminus B_1$ and some constant $C>0$.
\end{proposition}
\begin{proof}
The proof is based on the semigroup law and the symmetry of $k_\mu(t,\cdot ,\cdot)$ for $t>0$, and the semigroup law imply that
$$k_\mu(t,x,x)=\int_{\R^N}k_\mu(t/2,x,y)^2d\mu(y),\quad t>0,\,x\in \R^N.$$
By H\"older's inequality and \eqref{k-mu}, we deduce that
\begin{eqnarray*}
e^{\lambda_0\frac{t}{2}}\Phi(x)&=& T_2(t/2)\Phi(x)\\
&=& \int_{\R^N}k(t/2,x,y)\Phi(y)\,dy\\
&=& \phi(x)^{-1}\int_{\R^N}k_\mu(t/2,x,y)\phi(y)\Phi(y)\,d\mu(y)\\
&\le & \left(\int_{\R^N}k_\mu(t/2,x,y)^2d\mu(y)\right)^{\frac{1}{2}}\|\phi \Phi\|_{L^2_\mu}\phi(x)^{-1}\\
&=& k_\mu(t,x,x)^{\frac{1}{2}}\|\phi \Phi\|_{L^2_\mu}\phi(x)^{-1}
\end{eqnarray*}
for all $t>0$ and $x\in \R^N$. The assertion follows now from Proposition \ref{baondedofphi}.
\end{proof}

Now, in order to estimate $k_\mu$ we use the techniques in \cite[Chap 4]{Davies1}
\begin{proposition}\label{nash} There exist positive constants $ C_1,{C_2},C_3,C_4$ such that
   \begin{equation}\label{ip-2-rosen}
 \int _{\R^N} -\log (T\Phi) |v|^2d\mu\leq \eps h(v,v)+(C_1\eps^{-\gamma}+{C_2})\|v\|^2_{L^2_\mu},\quad v\in D(h),
\end{equation}
for any $\eps >0$ with $\gamma = \frac{\beta-\alpha+2}{\beta+\alpha-2}$, and
    \begin{equation}\label{ultra-mu}
\int_{\R^N}f|v|^2d\mu
	\leq C_3 \|f\|_{L_\mu^{N/2}}\left( h(v)+C_4\|v\|^2_{L^2_\mu}\right),\quad v\in D(h),\ f\in L^{N/2}_\mu ,
\end{equation}
\end{proposition}
  \begin{proof}

To prove \eqref{ip-2-rosen}, we apply the lower estimate of $\Phi$ obtained in Proposition \ref{baondedofphi}
\[
 C_1|x|^{-\frac{N-1}{2}-\frac{1}{4}(\beta-\alpha) }
    e^{-\frac{2}{\beta-\alpha+2}|x|^{\frac{\beta-\alpha+2}{2}}}  \le C_1|x|^{-\frac{N-1}{2}-\frac{1}{4}(\beta-\alpha) }
    e^{-\int_1^{|x|}\sqrt{\frac{r^\beta}{1+r^\alpha}dr}} \le (T\Phi)(x)
\]
we get
\begin{eqnarray*}
 -\log(T\Phi) &\leq &-\log C_1 +\left( \frac{N-1}{2} + \frac{\beta-\alpha}{4} \right)\log |x| +\frac{2}{\beta -\alpha +2}|x|^{\frac{\beta-\alpha+2} {2}}\\
\end{eqnarray*}
for $|x|\ge 1$.
Setting $\xi=\frac{\beta-\alpha}{2}+1$ we have
\begin{eqnarray*}
 -\log(T\Phi) &\leq &-\log C_1 +\frac{1}{2}\left( N-2+\xi \right)\log |x| +\frac{1}{\xi}|x|^{\xi}\\
\end{eqnarray*}
As a consequence, there are positive constants $C_2,\,C_3$
such that
$$-\log(T\Phi) \leq C_2\left( 1+\frac{1}{\xi} \right)|x|^{\xi}+C_3,\quad x\in \R^N.$$
Since $\xi<\beta,\,\gamma=\frac{\xi}{\beta -\xi}$, by using Young's inequality,\footnote{$xy\le \varepsilon x^p + C_\varepsilon y^q$, $\forall x\ge 0\,\, \forall y \ge 0$ with $C_\varepsilon = \varepsilon^{-\frac{1}{p-1}}$ and $\frac{1}{p}+ \frac{1}{q} =1$.} it follows that
\[
C_2\left( 1+\frac{1}{\xi} \right)|x|^{\xi}\leq \eps |x|^\beta+C_4\eps^{-\frac{\xi}{\beta-\xi}} = \eps V(x)+ C_4\eps^{-\gamma}
\]
for all $\eps >0$ where
$C_4=\left( C_2 \left(1+\frac{1}{\xi}\right) \right)^{\frac{\beta}{\beta-\xi}}$.
Thus,
\[
 -\log(T\Phi)\leq \eps |x|^\beta +C_4\eps^{-\gamma}+C_3.
\]
Taking into account that
\[
 0\leq \int_{\R^N}|\nabla v|^2dx+\int_{\R^N}\frac{\Delta \phi}{\phi}|v|^2dx=h(v,v)-\int_{\R^N}|x|^\beta v^2d\mu ,\quad v\in D(h),
\]
we obtain
\[
\int _{\R^N} -\log (T\Phi) |v|^2 d\mu\leq \eps h(v,v)+(C_4\eps^{-\gamma}+{C_3})\int_{\R^N}|v|^2d\mu .
\]
This proves \eqref{ip-2-rosen}.

Concerning (\ref{ultra-mu}), by density, it suffices to show it for $v\in C_c^\infty(\R^N)$. Using H\"older and Sobolev's inequality we obtain
\begin{eqnarray*}
\int_{\R^N}fv^2d\mu &\leq & \| f\|_{L^{N/2}_{\mu}} \| v\|^2_{L^{\frac{2N}{N-2}}_{\mu}}
    \leq  \| f\|_{L^{N/2}_{\mu}} \| v\|^2_{L^{\frac{2N}{N-2}}}\\
&\leq & \| f\|_{L^{N/2}_{\mu}} \| \nabla v\|^2_2.
\end{eqnarray*}
So, one obtains (\ref{ultra-mu}) by observing that
\[
 \| \nabla v\|^2_2=h(v,v)-\int_{\R^N}U|v|^2d\mu\leq h(v,v)+C_5 \|v\|_{L^2_\mu}^2,
\]
where $C_5=-\left( 0\wedge\min_{x\in \R^N}U\right)$.
\end{proof}
We give now the estimate of $k_\mu(t,x,y)$.
\begin{theorem} Assume that $N>2,\,\alpha >2$ and $\beta >\alpha -2$. Then there exist $C_1,C_2$ positive constants such that
\begin{equation}\label{eq:stima-kmu}
 k_{\mu}(t,x,y)\leq  C_1e^{C_2 t^{-\gamma}}(T\Phi)(x)(T\Phi)(y),\quad 0<t\le 1,\,x,\,y\in \R^N.
\end{equation}
\end{theorem}
\begin{proof}
It follows from Proposition \ref{nash} and Rosen's lemma, cf. \cite[Lemma 4.4.1]{Davies1}, that
for all $0\le f\in L^2(\R^N,(T\Phi)^2 \mu dx)$  and $\varepsilon >0$, we have
\begin{eqnarray*}
\int_{\R^N}(\log f)|v|^2d\mu &\leq & \varepsilon h(v,v)+({C_2}-\frac{N}{4}\log \varepsilon+\varepsilon \frac{C_4}{2}+C_1\varepsilon ^{-\gamma})\int_{\R^N}|v|^2d\mu \\
&\leq & \varepsilon h(v,v)+(C_5+C_6\varepsilon^{-\gamma})\int_{\R^N}|v|^2d\mu , \quad v\in D(h).
\end{eqnarray*}
So, applying \cite[Corollary 4.4.2]{Davies1}, one obtains
\begin{eqnarray*}
 \int_{\R^N} (f^2\log f) (T\Phi)^2d\mu &\leq & \varepsilon h_{T\Phi}(f,f)+\left( C_5+C_6\varepsilon^{-\gamma} \right)\|f\|^2_{L^2(\R^N,(T\Phi)^2 d\mu)}\\
    & & +\|f\|^2_{L^2(\R^N,(T\Phi)^2 d\mu)}\log \|f\|^2_{L^2(\R^N,(T\Phi)^2 d\mu)}
\end{eqnarray*}
for all $0\le f\in L^1\cap L^\infty \cap D(h_{T\Phi})$ and all $\varepsilon >0$, where $h_{T\Phi}$ is the quadratic form defined by
$h_{T\Phi}(v)=h((T\Phi) v,(T\Phi) v)$ for $v\in D(h_{T\Phi}):=\{v\in L^2(\R^N,(T\Phi)^2 d\mu): (T\Phi)v\in D(h)\}$.
We observe that $0<\gamma<1$. Then we can apply \cite[Corollary 2.2.8]{Davies1} and \cite[Lemma 2.1.2]{Davies1} to obtain that
\[
 0\leq k_{T\Phi}(t,x,y)\leq C_7e^{C_8t^{-\gamma}},\quad 0<t\le 1,\,x,\,y\in \R^N,
\]
where $k_{T\Phi}(\cdot ,\cdot ,\cdot)$ is the heat kernel of the semigroup generated by the selfadjoint operator associated to the form $h_{T\Phi}$.
\\
The result follows now by taking into account
\[
 k_\mu(t,x,y)=(T\Phi)(x)k_{T\Phi}(t,x,y)(T\Phi)(y),\quad t>0,\,x,\,y\in \R^N.
\]
\end{proof}

Now, we are ready to state and give the proof of the main result of this paper.
\begin{theorem}\label{heat-estimate}
If $N>2$, $\beta>\alpha-2$, $\alpha> 2$ then
\begin{align}\label{kernel0}
&k(t,x,y)\nonumber \\
& \leq C_1e^{\lambda_0 t+ C_2t^{-\gamma}}\left(\frac{1+|y|^\alpha}{1+|x|^\alpha}\right)^{\frac{b}{2\alpha}}
\frac{(|x||y|)^{-\frac{N-1}{2}-\frac{1}{4}(\beta-\alpha)}}{1+|y|^\alpha}
e^{-\int_1^{|y|}\sqrt{\frac{r^\beta}{1+r^\alpha}}dr}
    e^{-\int_1^{|y|}\sqrt{\frac{r^\beta}{1+r^\alpha}}dr}
\end{align}
for all $t>0,\,x,y\in \R^N\setminus B_1$, where $C_1, C_2$ are positive constants and $\gamma = \frac{\beta-\alpha+2}{\beta+\alpha-2}$.
\end{theorem}
\begin{proof}
By \eqref{k-mu}, \eqref{eq:stima-kmu} and Proposition \ref{baondedofphi} we have
\begin{eqnarray*}
k(t,x,y) &=& \frac{\phi(y)}{\phi(x)}\frac{1}{1+|y|^\alpha}k_\mu(t,x,y)\\
&\leq &
C_1\frac{\phi^2(y)}{1+|y|^\alpha}e^{-C_2t^{-\gamma}}\Phi(x)\Phi(y)\\
&=& C_1\frac{ (1+|y|^\alpha)^{\frac{b}{\alpha}} }{1+|y|^\alpha}e^{-C_2t^{-\gamma}}\Phi(x)\Phi(y)\\
&\le & C_1e^{-C_2t^{-\gamma}}\left(\frac{1+|y|^\alpha}{1+|x|^\alpha}\right)^{\frac{b}{2\alpha}}
\frac{(|x||y|)^{-\frac{N-1}{2}-\frac{1}{4}(\beta-\alpha)}}{1+|y|^\alpha}\\
& & \quad 
e^{-\int_1^{|y|}\sqrt{\frac{r^\beta}{1+r^\alpha}}dr}
    e^{-\int_1^{|y|}\sqrt{\frac{r^\beta}{1+r^\alpha}}dr}
\end{eqnarray*}
for $0<t \le 1,\,x,\,y\in \R^N\setminus B_1$ and $\gamma = \frac{\beta-\alpha+2}{\beta+\alpha-2}$. 

To end the proof let us consider the case $t>1$.

The semigroup law and the symmetry of $k_\mu(t,\cdot ,\cdot)$ imply that
\begin{eqnarray*}
k_{\mu}(t,x,y)=\int_{\R^N}k_{\mu}(t-1/2,x,z)k_{\mu}(1/2,y,z)d\mu(z),\qquad\,\,t>1/2,\;\,x,y\in\R^N.
\end{eqnarray*}
On the other hand, thanks to
 \eqref{eq:stima-kmu}, and since $\beta>\alpha-2$, one deduces that the function $k_{\mu}(1/2,y,\cdot)$ belongs to $L^2_{\mu}$. Hence,
\begin{eqnarray*}
k_{\mu}(t,x,y)=(e^{(t-\frac{1}{2})H_{\mu}}k_{\mu}(1/2,y,\cdot))(x),\qquad\;\,t>1/2,\;\,x,y\in\R^N.
\end{eqnarray*}
Using again the semigroup law and the symmetry we have
\begin{align*}
k_{\mu}(t,x,x)=&\int_{\R^N}|k_{\mu}(t/2,x,y)|^2d\mu(y)\\
\le & Me^{\lambda_0(t-1)}\|k_{\mu}(1/2,x,\cdot)\|_{L^2_{\mu}}^2\\
= & Me^{\lambda_0(t-1)}k_{\mu}(1,x,x),\quad t>1,\,x\in \R^N.
\end{align*}
So, by applying (\ref{eq:stima-kmu}) to $k_{\mu}(1,x,x)$ and using the inequality 
$$k_{\mu}(t,x,y)\le (k_{\mu}(t,x,x))^{1/2}(k_{\mu}(t,y,y))^{1/2},\quad t>0,\,x,\,y\in \R^N,$$ 
one obtains (\ref{kernel0}).
\end{proof}
\begin{remark}
\begin{itemize}
\item[1.] The heat kernel estimates $k(\cdot,\cdot,\cdot)$ in Theorem \ref{heat-estimate} could be sharp in the space variables. This can be seen from Proposition \ref{on-diagonal}.
\item[2.] Since, for $r\ge 1$, $\sqrt{\frac{2r^\beta}{1+r^\alpha}} \ge r^{\frac{\beta -\alpha}{2}}$, it follows from Theorem \ref{heat-estimate} that
$$
k(t,x,y) \leq c_1e^{\lambda_0 t+ c_2t^{-\gamma}}\left(\frac{1+|y|^\alpha}{1+|x|^\alpha}\right)^{\frac{b}{2\alpha}}
\frac{(|x||y|)^{-\frac{N-1}{2}-\frac{1}{4}(\beta-\alpha)}}{1+|y|^\alpha}
e^{-\frac{\sqrt{2}}{\beta-\alpha+2}\left(|x|^{\frac{\beta-\alpha+2}{2}}+ |y|^{\frac{\beta-\alpha+2}{2}}\right)}
$$
for $t>0,\,|x|,\,|y|\ge 1$.
\end{itemize}
\end{remark}

If we denote by $\Phi_j$ the eigenfunction of $A_2$ associated to the eigenvalue $\lambda_j$, then $T\Phi_j$ is the eigenfunction of $H_\mu$ associated to $\lambda_j$. Hence, for any $t>0$ and any $x\in \R^N$, we have
\begin{eqnarray*}
e^{\lambda_jt}|T\Phi_j(x)| &=&\left|\int_{\R^N}k_\mu(t,x,y)(T\Phi_j)(y)d\mu(y)\right|\\
&\le & \left(\int_{\R^N}k_\mu(t,x,y)^2d\mu(y)\right)^{\frac{1}{2}}\|T\Phi_j\|_{L^2_\mu}\\
&=& (k_\mu(2t,x,x))^{\frac{1}{2}}\|T\Phi_j\|_{L^2_\mu}.
\end{eqnarray*}
So, by \eqref{eq:stima-kmu}, we obtain the following estimates.
\begin{coro}
Suppose that the assumptions of Theorem \ref{heat-estimate} hold. Then all eigenfunctions $\Phi_j$ of $A$ with $\|T\Phi_j\|_{L^2_\mu}=1$ satisfy
$$
|\Phi_j(x)|\le C_j|x|^{\frac{\alpha-\beta}{4}-\frac{N-1}{2}} e^{-\int_1^{|x|}\sqrt{\frac{r^\beta}{1+r^\alpha}}dr},
$$
for all $j\in\N ,\,x\in \R^N\setminus B_1$ and some constant $C_j>0$.
\end{coro}

\begin{remark}
In the case $b>2-N$, we can obtain better estimates of the kernels $k$ with respect to the time variable $t$ for small $t$. In fact if we denote by $S(\cdot)$ the semigroup generated by $(1+|x|^\alpha)\Delta + b|x|^{\alpha-2}x\cdot \nabla$ in $C_b(\R^N)$, which is given by a kernel $p$, then by domination we have $0<k(t,x,y)\le p(t,x,y)$ for $t>0$ and $x,\,y\in \R^N$. So, by \cite[Remark 9.12]{Me-Sp-Ta}, it follows that
\begin{eqnarray}\label{eq-improved}
k(t,x,y) &\le & Ct^{-\frac{N+b+\alpha-4}{\alpha-2}}(1+|x|)^{2-N-b}(1+|y|)^{2-N-\alpha},\quad 2<\alpha\le 4+\frac{2b}{N-2}, \nonumber\\
k(t,x,y) &\le & Ct^{-N/2}(1+|x|^\alpha)^{2-N-b}(1+|y|^\alpha)^{2-N-\alpha},\quad \alpha\ge 4+\frac{2b}{N-2} \nonumber
\end{eqnarray}
for $0<t\le 1,\,x,\,y\in \R^N$.
\end{remark}



\begin{thebibliography}{99}
\bibitem{boutiah et al} {S.E. Boutiah, F. Gregorio, A. Rhandi, And C. Tacelli}, \emph{Elliptic operators with unbounded diffusion, drift and potential terms}, J. Differential Equations, 264 (2018), no. 3, 2184-2204, https://doi.org/10.1016/j.jde.2017.10.020.
\bibitem{Bakry} {D. Bakry, F. Bolley, I. Gentil, and P. Maheux}, \emph{Weighted Nash inequalities}, Rev. Mat. Iberoam. 28 (2012), no. 3, 879-906.
\bibitem{Davies1} {E.B. Davies}, \emph{ Heat kernels and spectral theory}, Cambridge University Press, Cambridge, 1989.
\bibitem{AC-AR-CT} {A. Canale, A. Rhandi, C. Tacelli}, \emph{Schr\"odinger-type operators with unbounded diffusion and potential terms}, Ann. Sc. Norm. Super. Pisa CI. Sci. (5) Vol. XVI (2016), 581-601.
\bibitem{AC2-AR2-CT2} {A. Canale, A. Rhandi, C. Tacelli}, \emph{Kernel estimates for Schrodinger type operators with unbounded diffusion and potential terms}, J. Anal. Appl. 36 (2017) 377-392, https://doi.org/10.4171/ZAA/1593.
\bibitem{AC-CT}{A. Canale, C. Tacelli}, \emph{Optimal kernel estimates for a Schrödinger type operator}, Riv. Mat. Univ. Parma Vol 7 (2016), 341-450.
\bibitem{F-L} {S. Fornaro, L. Lorenzi}, \emph{Generation results for elliptic operators with unbounded diffusion coefficients in $L^{p}$ and $C_{b}$-spaces}, Discrete and Continuous Dynamical Systems A18 (2007),747-772.
\bibitem{KLR} { M. Kunze, L. Lorenzi, A. Rhandi}, \emph{Kernel estimates for nonautonomous Kolmogorov equations}, Advances in Mathematics 287 (2016), 600-639.
\bibitem{DG-NT} {D. Gilbarg, N. Trudinger}, \emph{Elliptic Partial Differential Equations of Second Order}, Second edition, Springer, Berlin, (1983).
\bibitem{TD-RM-CT} {T. Durante, R. Manzo, C. Tacelli}, \emph{Kernel estimates for Schr\"odinger type operators with unbounded coefficients and critical exponent}, Ricerche Mat. Vol. 65 (2016),  289-305.
\bibitem{Lo-Be} {L. Lorenzi, M. Bertoldi}, \emph{Analytical Methods for Markov Semigroups}, Chapman $ \&$ Hall/CRC, (2007).
\bibitem{Luca - Abde} {L. Lorenzi, A. Rhandi}, \emph{On Schr\"odinger type operators with unbounded coefficients: generation and heat kernel estimates}, J. Evol. Equ. 15(2015), 53-88.
\bibitem{Me-Pa-Wa} {G. Metafune, D. Pallara, M. Wacker}, \emph{Feller Semigroups on $\mathbb{R}^{N}$}, Semigroup Forum 65 (2002), 159-205.
\bibitem{G-S} {G. Metafune, C. Spina}, \emph{Elliptic operators with unbounded coefficients in $L^{p}$ spaces}, Annali Scuola Normale Superiore di Pisa Cl. Sc. (5), 11 (2012), 303-340 .
\bibitem{G-S 3} {G. Metafune, C. Spina}, \emph{Kernel estimates for some elliptic operators with unbounded coefficients}, Discrete and Continuous Dynamical Systems A32 (6) (2012), 2285-2299.
\bibitem{Me-Sp-Ta} {G. Metafune, C. Spina, C. Tacelli}, \emph{Elliptic operators with unbounded diffusion and drift coefficients in $L^{p}$ spaces}, Adv. Diff. Equat. 19 (2014), no. 5-6, 473-526.
\bibitem{Me-Sp-Ta2} {G. Metafune, C. Spina, C. Tacelli}, \emph{On a class of elliptic operators with unbounded diffusion coefficients}, Evol. Equ. Control Theory 3, (2014), no. 4, 671-680.
\bibitem{Olver}{F.W.J. Olver}, \emph{Asymptotics and special functions}, Academic Press, New York, 1974.
\bibitem{Ouhabaz}{E.M. Ouhabaz}, \emph{Analysis of heat equations on domains}, London Math. Soc. Monogr. Ser., 31, Princeton Univ. Press, 2004.
\end{thebibliography}
\end{document}